\documentclass[12pt,a4paper,reqno]{amsart} 
\usepackage{amssymb}
\usepackage{latexsym}
\usepackage{amsmath}
\usepackage{mathrsfs}
\usepackage{enumerate}
\usepackage{euscript}
\usepackage{calc}                         
\usepackage{cite}

\newcommand{\scal}[2]{\langle #1,#2\rangle}

\newcommand{\rr}[1]{\mathbf R^{#1}}

\newcommand{\nm}[2]{\Vert #1\Vert _{#2}}
\newcommand{\nmm}[1]{\Vert #1\Vert }
\newcommand{\abp}[1]{\vert #1\vert}

\newcommand{\op}{\operatorname{Op}}

\newcommand{\sets}[2]{\{ {\,}#1{\,};{\,}#2{\,}\} }
\newcommand{\ep}{\varepsilon}

\newcommand{\cdo}{\, \cdot \, }

\newcommand{\wpr}{{\text{\footnotesize{$\#$}}}}

\newcommand{\eabs}[1]{\langle #1\rangle}

\newcommand{\vrum}{\vspace{0.1cm}}

\newcommand{\GL}{\mathbf{M}}

\newcommand{\masfQ}{\mathsf Q}
\newcommand{\masfR}{\mathsf R}

\newcommand{\maclS}{\mathcal S}
\newcommand{\maclV}{\mathcal V}

\newcommand{\mascF}{\mathscr F}
\newcommand{\mascS}{\mathscr S}
\newcommand{\mascP}{\mathscr P}
\newcommand{\splM}{\EuScript M}
\newcommand{\splW}{\EuScript W}

\numberwithin{equation}{section}          

\newtheorem{thm}{Theorem}
\numberwithin{thm}{section}

\newcommand{\rubrik}{}
\newtheorem{prop}[thm]{Proposition}

\newtheorem{lemma}[thm]{Lemma}

\theoremstyle{definition}

\newtheorem{defn}[thm]{Definition}

\theoremstyle{remark}

\newtheorem{rem}[thm]{Remark}              

\title{Multi-linear products with
odd factors in pseudo-differential
calculus with symbols in modulation spaces}

\author{Joachim Toft}

\address{Department of Mathematics, Linn{\ae}us University, V{\"a}xj{\"o},
Sweden}

\email{joachim.toft@lnu.se}


\begin{document}

\begin{abstract}
We give sufficient conditions on the Lebesgue exponents for
compositions of odd numbers of pseudo-differential
operators with symbols in modulation spaces. As a byproduct,
we obtain sufficient conditions for twisted convolutions of odd
numbers of factors to be bounded on Wiener amalgam spaces.
\end{abstract}

\maketitle

\section{Introduction}\label{sec0}

\par


In the paper we deduce multi-linear Weyl products
and other products in pseudo-differential calculus
of odd factors with symbols belonging to suitable
modulation spaces. Especially we improve some
of the odd multi-linear products in \cite{CoToWa}. Here we
notice that for corresponding bilinear products, the results in
\cite{CoToWa} are sharp.

\par

Suppose that $N\ge 1$ is an integer. Then
it follows from \cite[Proposition 2.5]{CoToWa} that
\begin{gather}
M^{p_1,q_1}\wpr M^{p_2,q_2}\wpr M^{p_3,q_3} \wpr
\cdots \wpr M^{p_N,q_N}
\subseteq
M^{p_0',q_0'}
\label{Eq:NonWeightModEmb}
\intertext{holds true when}
\max \left ( \masfR _N(\textstyle{\frac 1{q'}}) ,0\right )
\le
\min \left ( 
{\textstyle{\frac 1{p_j},\frac 1{p_j'},\frac 1{q_j},\frac 1{q_j'}},\masfR _N(\frac 1p) }
\right ).
\label{Eq:LebExpCondProp2.5}
\end{gather}
Here
$$
\masfR _N (x) = \frac 1{N-1}\left (
\sum _{j=0}^Nx_j -1
\right ),
\qquad x=(x_0,x_1,\dots ,x_N)\in [0,1]^{N+1},
$$
and we observe that \cite[Proposition 2.5]{CoToWa}
is a weighted version of \eqref{Eq:NonWeightModEmb}
and \eqref{Eq:LebExpCondProp2.5}.
(See \cite{Ho3} and Section \ref{sec1} for notations.) We notice
that the conditions on $\frac 1{p_j'}$ and $\frac 1{q_j}$ in
\eqref{Eq:LebExpCondProp2.5} can be
outlined because they are always fulfilled when the other conditions
hold (see e.{\,}g. Theorem 0.1$'$ in \cite{CoToWa} and its proof).

\par

The multi-linear multiplication property \eqref{Eq:NonWeightModEmb}
with \eqref{Eq:LebExpCondProp2.5} is obtained by interpolating the case
that \eqref{Eq:NonWeightModEmb} holds true when
\eqref{Eq:LebExpCondProp2.5} is replaced by
\begin{equation}\tag*{(\ref{Eq:LebExpCondProp2.5})$'$}
\masfR _N(\textstyle{\frac 1{q'}}) \le 0 \le \masfR _N(\textstyle{\frac 1p}),
\end{equation}
(see Proposition 2.1 in \cite{CoToWa}) with the case
\begin{equation}\tag*{(\ref{Eq:NonWeightModEmb})$'$}
M^{2,2}\wpr M^{2,2}\wpr M^{2,2}\wpr \cdots \wpr M^{2,2}
\subseteq M^{2,2}.
\end{equation}
Note that the last the last multiplication property is equivalent with
the fact that compositions of Hilbert-Schmidt operators result
into new Hilbert-Schmidt operators.

\par

In the paper we improve \eqref{Eq:NonWeightModEmb}
and \eqref{Eq:LebExpCondProp2.5} for multi-linear products
with odd factors, by replacing \eqref{Eq:NonWeightModEmb}$'$
with the more general
\begin{equation}\tag*{(\ref{Eq:NonWeightModEmb})$''$}
M^{p,p}\wpr M^{p',p'}\wpr M^{p,p}\wpr \cdots \wpr M^{p,p}
\subseteq M^{p,p}
\end{equation}
in the interpolation with \eqref{Eq:NonWeightModEmb}
and \eqref{Eq:LebExpCondProp2.5}$'$. This leads to
that \eqref{Eq:NonWeightModEmb} remains true for odd $N$,
when \eqref{Eq:LebExpCondProp2.5} is replaced by
\begin{equation}\tag*{(\ref{Eq:LebExpCondProp2.5})$''$}
\max \left ( \masfR _N(\textstyle{\frac 1{q'}}) ,0\right )
\le
\min \left ( 
{\textstyle{\masfQ _N(\frac 1{p}),\masfQ _N(\frac 1{q'}),
\masfQ _N(\frac 1{p},\frac 1{q'}),
\masfR _N(\frac 1p) }} \right ).
\end{equation}
Here
$$
\masfQ _N (x,y) = \min _{j+k\, \text{odd}}(\textstyle{\frac 12}(x_j+y_k),
1-\textstyle{\frac 12}(x_j+y_k))
\quad \text{and}\quad
\masfQ _N (x) = \masfQ _N (x,x),
$$
when
$$
x=(x_0,x_1,\dots ,x_N)\in [0,1]^{N+1}
\quad \text{and}\quad
y=(y_0,y_1,\dots ,y_N)\in [0,1]^{N+1}.
$$
In fact, in Section \ref{sec2} we deduce weighted
versions of \eqref{Eq:NonWeightModEmb} 
and \eqref{Eq:LebExpCondProp2.5}$'$ (see Proposition
\ref{Prop:MainPropOdd} and Theorem
\ref{Thm:MainThmOdd} in Section
\ref{sec2}).

\medspace

We observe that \eqref{Eq:NonWeightModEmb} is close
to certain regularization techniques for linear operators in
distribution theory.
In fact, a convenient way to regularizing operators might be to
enclose them with regularizing operators. For example,
suppose that $T$ is a linear and continuous operator
from the Schwartz space $\mascS (\rr d)$ to the
set of tempered distributions $\mascS '(\rr d)$.
The operator $T$ possess weak continuity
properties in the sense that it is only guaranteed that
$T$ maps elements from $\mascS (\rr d)$ into the
significantly larger space $\mascS '(\rr d)$. On the other hand,
by enclosing $T$ with operators $S_1,S_2$
which are regularizing in the sense that they are
mapping $\mascS '(\rr d)$ into the smaller space $\mascS (\rr d)$,
the resulting operator
$$
T_0=S_1\circ T\circ S_2
$$
becomes again regularizing. Equivalently, by using the kernel theorem of
Schwartz and identifying kernels with operators one has that
\begin{equation}\label{Eq:OpThreeKernels}
\begin{gathered}
(K_1,K_2,K_3) \mapsto K_1\circ K_2\circ K_3,
\\[1ex]
\mascS (\rr {2d})\times \mascS '(\rr {2d})\times \mascS (\rr {2d})
\to
\mascS (\rr {2d}),
\end{gathered}
\end{equation}
is sequently continuous. Here we observe that by omitting one of
the compositions one may only guarantee that
$$
(K_1,K_2)\mapsto K_1\circ K_2
$$
is sequently continuous from
$\mascS (\rr {2d})\times \mascS '(\rr {2d})$
or
$\mascS '(\rr {2d})\times \mascS (\rr {2d})$
to $\mascS '(\rr {2d})$, which is a significantly weaker
continuity property. 

\par

The continuity of the map \eqref{Eq:OpThreeKernels} can be obtained
by rewriting $K_1\circ K_2\circ K_3$ as
\begin{align*}
K_1\circ K_2\circ K_3(x,y) &= R_{K_2,F}(x,y)\equiv
\scal {K_2}{F(x,y,\cdo )},
\intertext{where}
F(x,y,x_1,x_2) &= (K_1\otimes K_3)(x,x_1,x_2,y)
=
K_1(x,x_1)K_3(x_2,y).
\end{align*}
The asserted continuity then follows from the facts
that $(K,F)\mapsto R_{K,F}$ is
continuous from $\mascS '(\rr {2d})\times \mascS (\rr {4d})$
to $\mascS (\rr {2d})$, and that $(K_1,K_3)\mapsto K_1\otimes K_3$
is continuous from $\mascS (\rr {2d})\times \mascS (\rr {2d})$
to $\mascS (\rr {4d})$.

\par

The mapping properties can also be conveniently formulated
within the theory of pseudo-differential calculus, e.{\,}g. in the Weyl calculus.
Recall that the Weyl product $\wpr$ of the elements
$a_1,a_2\in \mascS (\rr {2d})$ are defined by the formula
$$
\op ^w(a_1\wpr a_2) = \op ^w(a_1)\circ \op ^w(a_2),
$$
where $\op ^w(a)$ is the Weyl operator of $a\in \mascS (\rr {2d})$,
defined by
$$
\op ^w(a)f(x) = (2\pi)^{-d}\iint _{\rr {2d}}a({\textstyle{\frac 12}}(x+y),\xi )
f(y)e^{i\scal {x-y}\xi}\, dyd\xi ,
$$
when $f\in \mascS (\rr d)$. The definition of $\op ^w(a)$ extends
to a continuous operator from $\mascS '(\rr d)$ to $\mascS (\rr d)$.
We may also extend the definition of $\op ^w(a)$ to allow
$a\in \mascS '(\rr {2d})$, and then $\op ^w(a)$ is continuous from
$\mascS (\rr d)$ to $\mascS '(\rr d)$.

\par

By straight-forward Fourier techniques it follows that
\eqref{Eq:OpThreeKernels} is equivalent to
\begin{align}
\mascS (\rr {2d})\wpr \mascS '(\rr {2d})\wpr \mascS (\rr {2d})
&\subseteq
\mascS (\rr {2d}),
\label{Eq:ThreeSymbolsSchwartz}
\intertext{which by duality gives}
\mascS '(\rr {2d})\wpr \mascS (\rr {2d})\wpr \mascS '(\rr {2d})
&\subseteq
\mascS '(\rr {2d}),
\label{Eq:ThreeSymbolsSchwartzDual}
\end{align}
Again we observe that for the bilinear case we only have
$$
\mascS (\rr {2d})\wpr \mascS '(\rr {2d})\subseteq
\mascS '(\rr {2d}),
$$
and that $\mascS '(\rr {2d})\wpr \mascS '(\rr {2d})$ does not
make any sense.

\par

By similar arguments, \eqref{Eq:ThreeSymbolsSchwartz}
and \eqref{Eq:ThreeSymbolsSchwartzDual} can be extended
to any odd-linear Weyl products. That is, one has
\begin{align}
\mascS (\rr {2d})\wpr \mascS '(\rr {2d})\wpr \mascS (\rr {2d})
\wpr \cdots \wpr \mascS (\rr {2d})
&\subseteq
\mascS (\rr {2d}),
\tag*{(\ref{Eq:ThreeSymbolsSchwartz})$'$}
\intertext{and}
\mascS '(\rr {2d})\wpr \mascS (\rr {2d})\wpr \mascS '(\rr {2d})
\wpr \cdots \wpr \mascS '(\rr {2d})
&\subseteq
\mascS '(\rr {2d}),
\tag*{(\ref{Eq:ThreeSymbolsSchwartzDual})$'$}
\end{align}
which are analogous to \eqref{Eq:NonWeightModEmb}$''$,
taking into account that $M^{p',p'}$ is the dual of $M^{p,p}$
when $p<\infty$.
%
%

\par


\par

\section{Preliminaries}\label{sec1}

\par

In this section we introduce notation and discuss the background on 
Gelfand--Shilov spaces, 
pseudo-differential operators, the Weyl
product, twisted convolution and modulation spaces.
Most proofs can be found in the literature and are therefore omitted. 

\par

Let $0<h,s\in \mathbf R$ be fixed. The space $\mathcal S_{s,h}(\rr d)$
consists of all $f\in C^\infty (\rr d)$ such that
\begin{equation*}
\nm f{\mathcal S_{s,h}}\equiv \sup \frac {|x^\beta \partial ^\alpha
f(x)|}{h^{|\alpha | + |\beta |}\alpha !^s\, \beta !^s}
\end{equation*}
is finite, with supremum taken over all $\alpha ,\beta \in
\mathbf N^d$ and $x\in \rr d$.

\par

The space $\mathcal S_{s,h} \subseteq
\mathscr S$ ($\mathscr S$ denotes the Schwartz space) is a Banach space
which increases with $h$ and $s$.
Inclusions between topological spaces are understood to be continuous. 
If $s>1/2$, or $s =1/2$ and $h$ is sufficiently large, then $\mathcal
S_{s,h}$ contains all finite linear combinations of Hermite functions.
Since the space of such linear combinations is dense in $\mathscr S$, it follows
that the topological dual $(\mathcal S_{s,h})'(\rr d)$ of $\mathcal S_{s,h}(\rr d)$ is
a Banach space which contains $\mathscr S'(\rr d)$.

\par

\subsection{Gelfand-Shilov spaces of functions and distributions}

\par

The \emph{Gelfand--Shilov spaces} $\mathcal S_{s}(\rr d)$ and
$\Sigma _{s}(\rr d)$ (cf. \cite{GS}) are the inductive and projective limits, respectively,
of $\mathcal S_{s,h}(\rr d)$, with respect to the parameter $h$. Thus
\begin{equation}\label{GSspacecond1}
\mathcal S_{s}(\rr d) = \bigcup _{h>0}\mathcal S_{s,h}(\rr d)
\quad \text{and}\quad \Sigma _{s}(\rr d) =\bigcap _{h>0}\mathcal S_{s,h}(\rr d),
\end{equation}
where $\mathcal S_{s}(\rr d)$ is equipped with the the strongest topology such
that the inclusion map from $\mathcal S_{s,h}(\rr d)$ into $\mathcal S_{s}(\rr d)$
is continuous, for every choice of $h>0$. The space $\Sigma _s(\rr d)$ is a
Fr{\'e}chet space with seminorms
$\nm \cdo{\mathcal S_{s,h}}$, $h>0$. We have $\Sigma _s(\rr d)\neq \{ 0\}$
if and only if $s>1/2$, and $\maclS _s(\rr d)\neq \{ 0\}$
if and only if $s\ge 1/2$. From now on we assume that $s>1/2$ when we
consider $\Sigma _s(\rr d)$, and $s\ge 1/2$ when we consider $\maclS
_s(\rr d)$.

\medspace

The \emph{Gelfand--Shilov distribution spaces} $\mathcal S_{s}'(\rr d)$
and $\Sigma _s'(\rr d)$ are the projective and inductive limit
respectively of $\mathcal S_s'(\rr d)$.  This means that
\begin{equation}\tag*{(\ref{GSspacecond1})$'$}
\mathcal S_s'(\rr d) = \bigcap _{h>0}\mathcal S_{s,h}'(\rr d)\quad
\text{and}\quad \Sigma _s'(\rr d) =\bigcup _{h>0} \mathcal S_{s,h}'(\rr d).
\end{equation}
In \cite{GS, Ko, Pil} it is proved that $\mathcal S_s'(\rr d)$
is the topological dual of $\mathcal S_s(\rr d)$, and $\Sigma _s'(\rr d)$
is the topological dual of $\Sigma _s(\rr d)$.

\par

For each $\ep >0$ and $s>1/2$ we have
\begin{equation}\label{GSembeddings}
\begin{alignedat}{3}
\maclS _{1/2}(\rr d) & \subseteq &\Sigma _s (\rr d) & \subseteq&
\maclS _s(\rr d) & \subseteq \Sigma _{s+\ep}(\rr d)
\\[1ex]
\quad \text{and}\quad
\Sigma _{s+\ep}' (\rr d) & \subseteq  & \maclS _s'(\rr d)
& \subseteq & \Sigma _s'(\rr d) & \subseteq \maclS _{1/2}'(\rr d).
\end{alignedat}
\end{equation}

\par

The Gelfand--Shilov spaces are invariant under several basic operations,
e.{\,}g. translations, dilations, tensor products
and (partial) Fourier transformation.

\par

We normalize the Fourier transform of $f\in L^1(\rr d)$ as
$$
(\mathscr Ff)(\xi )= \widehat f(\xi ) \equiv (2\pi )^{-\frac d2}\int _{\rr
{d}} f(x)e^{-i\scal  x\xi }\, dx,
$$
where $\scal \cdo \cdo$ denotes the scalar
product on $\rr d$. The map $\mathscr F$ extends
uniquely to homeomorphisms on $\mathscr S'(\rr d)$, $\mathcal S_s'(\rr d)$
and $\Sigma _s'(\rr d)$, and restricts to
homeomorphisms on $\mathscr S(\rr d)$, $\mathcal S_s(\rr d)$ and
$\Sigma _s(\rr d)$, and to a unitary operator on $L^2(\rr d)$.

\par

\subsection{Modulation spaces}

\par

Next we turn to the basic properties of modulation spaces, and start by
recalling the conditions for the involved weight functions. Let
$0<\omega \in L^\infty _{loc}(\rr d)$. Then $\omega$ is called
\emph{moderate}  if there is a function $0<v\in L^\infty _{loc}(\rr d)$
such that
\begin{equation}\label{moderate}
\omega (x+y) \lesssim \omega (x)v(y),\quad x,y\in \rr d.
\end{equation}
Then $\omega$ is also called \emph{$v$-moderate}.
Here the notation $f(x) \lesssim g(x)$ means that there exists $C>0$
such that $f(x) \leq C g(x)$ for all arguments $x$ in the domain of $f$
and $g$. If $f \lesssim g$ and $g \lesssim f$ we write $f \asymp g$. 
The function $v$ is called \emph{submultiplicative}
if it is even and \eqref{moderate} holds when $\omega =v$. We note that if
\eqref{moderate} holds then
$$
v^{-1}\lesssim \omega  \lesssim v.
$$
For such $\omega$ it follows that \eqref{moderate} is true when
$$
v(x) =Ce^{c|x|},
$$
for some positive constants $c$ and $C$ (cf. \cite{Grochenig5}).
In particular, if $\omega$ is moderate on $\rr d$, then
$$
e^{-c|x|}\lesssim \omega (x)\lesssim e^{c|x|},
$$
for some $c>0$.

\par

The set of all moderate functions on $\rr d$
is denoted by $\mascP _E(\rr d)$.
If $v$ in \eqref{moderate}
can be chosen as  $v(x)=\eabs{x}^s=(1+|x|^2)^{s/2}$ for some
$s \ge 0$, then $\omega$ is
said to be of polynomial type or polynomially moderate. We let
$\mascP (\rr d)$ be the set
of all polynomially moderate functions on $\rr d$.

\medspace

Let $\phi \in \maclS _s (\rr d) \setminus 0$ be
fixed. The \emph{short-time Fourier transform} (STFT) $V_\phi f$ of $f\in
\maclS _s ' (\rr d)$ with respect to the \emph{window function} $\phi$ is
the Gelfand--Shilov distribution on $\rr {2d}$ defined by
$$
V_\phi f(x,\xi ) \equiv 
\mascF (f \, \overline {\phi (\cdo -x)})(\xi ).
$$

\par

For $a \in \maclS _{1/2} '(\rr {2d})$ and
$\Phi \in \maclS _{1/2} (\rr {2d})  \setminus 0$ the
\emph{symplectic short-time Fourier transform} $\maclV _{\Phi} a$
of $a$ with respect to $\Phi$ is the defined similarly as
\begin{equation}\nonumber
\maclV _{\Phi} a(X,Y) = \mascF _\sigma \big( a\, \overline{ \Phi (\cdo -X) }
\big) (Y),\quad X,Y \in \rr {2d}.
\end{equation}
We have
\begin{multline}\label{stftcompare}
\maclV _{\Phi} a(X,Y) = 2^d V_\Phi a(x,\xi , -2\eta ,2y),
\\[1ex]
X=(x,\xi )\in \rr {2d},\ Y=(y,\eta )\in \rr {2d},
\end{multline}
which shows the close connection between $V_\Phi a$
and $\maclV _{\Phi} a$. 
The Wigner distribution $W_{f,\phi}$ and $V_\phi f$ are also closely related.

\par

If $f ,\phi \in \maclS _s (\rr d)$ and $a,\Phi \in \maclS _s (\rr {2d})$ then
\begin{align*}
V_\phi f(x,\xi ) &= (2\pi )^{-\frac d2}\int f(y)\overline {\phi
(y-x)}e^{-i\scal y\xi}\, dy 
\intertext{and}
\maclV _\Phi a(X,Y ) &= \pi ^{-d}\int a(Z)\overline {\Phi
(Z-X)}e^{2i\sigma (Y,Z)}\, dZ .
\end{align*}

\par

Let $\omega \in \mascP _E (\rr {2d})$, $p,q\in [1,\infty ]$
and $\phi \in \maclS _{1/2} (\rr d)\setminus 0$ be fixed. The
\emph{modulation space} $M^{p,q}_{(\omega )}(\rr d)$ consists of
all $f\in \maclS _{1/2} '(\rr d)$ such that
\begin{align}
\nm f{M^{p,q}_{(\omega )}} &\equiv \Big (\int _{\rr d}\Big (\int _{\rr d}
|V_\phi f(x,\xi )\omega (x,\xi )|^p\, dx\Big )^{q/p}\, d\xi \Big )^{1/q}
\label{modnorm1}
\intertext{is finite, and the Wiener amalgam
space $W^{p,q}_{(\omega )} (\rr d)$ consists of all $f\in
\maclS _{1/2} '(\rr d)$ such that}
\nm f {W^{p,q}_{(\omega )}} &\equiv \Big (\int _{\rr d}\Big (\int _{\rr d}
|V_\phi f(x,\xi )\omega (x,\xi )|^q\, d\xi \Big )^{p/q}\, dx \Big )^{1/p}
\label{modnorm2}
\end{align}
is finite (with obvious modifications in \eqref{modnorm1} and
\eqref{modnorm2} when $p=\infty$ or $q=\infty$).

\par

\begin{rem}\label{MoreWeightClasses}
As follows from Proposition \ref{p1.4} (2) below we have that in fact
$M^{p,q}_{(\omega )}(\rr d)$ contains the superspace $\Sigma _1(\rr d)$
of $\maclS _{1/2}(\rr d)$, and is contained in the subspace $\Sigma _1'(\rr d)$
of $\maclS _{1/2}'(\rr d)$, when $\omega \in \mascP _E(\rr {2d})$. Hence we
could from the beginning have assumed that $f \in \Sigma _1'(\rr d)$ in
\eqref{modnorm1} and \eqref{modnorm2}.

\par

On the other hand, in \cite{Toft8}, certain weight classes containing $\mascP
_E(\rr {2d})$ and superexponential weights are introduced. For any $s>1/2$,
the corresponding families of modulation spaces are large enough to contain
superspaces of $\maclS _s'(\rr d)$ and subspaces of $\maclS _s(\rr d)$.

\par

However, we are not dealing with these large families of modulation spaces
because we need (1) and (2) in Proposition \ref{p1.4} which are not known to be
true for weights of this generality.
\end{rem}

\par

\begin{rem}
The literature contains slightly different conventions
concerning modulation and Wiener amalgam spaces. Sometimes our
definition of a Wiener amalgam space is considered as a particular
case of a general class of modulation spaces (cf.
\cite{Feichtinger1,Feichtinger2,Feichtinger6}). Our definition is
adapted to give the relation \eqref{twistfourmod} that suits our
purpose to transfer continuity for the Weyl product on
modulation spaces to continuity for twisted convolution on Wiener
amalgam spaces. 
\end{rem}

\par

On the even-dimensional phase space $\rr {2d}$ we may define
modulation spaces based on the symplectic STFT. 
Thus if $\omega \in \mascP _E (\rr {4d})$, $p,q\in [1,\infty ]$
and $\Phi \in \maclS _{1/2} (\rr {2d})\setminus 0$ are fixed, then
the \emph{symplectic modulation spaces}
$\splM ^{p,q}_{(\omega )}(\rr {2d})$ and Wiener amalgam spaces $\splW ^{p,q}
_{(\omega )}(\rr {2d})$ are obtained by replacing
the STFT $a\mapsto V_\Phi a$ by the corresponding
symplectic version $a\mapsto \maclV _\Phi a$ in \eqref{modnorm1} and
\eqref{modnorm2}.
(Sometimes the word \emph{symplectic} before modulation space is
omitted for brevity.) 
By \eqref{stftcompare} we have
$$
\splM ^{p,q}_{(\omega )}(\rr {2d}) = M ^{p,q}_{(\omega _0)}(\rr {2d}),
\quad \omega(x,\xi, y, \eta) = \omega_0 (x,\xi, -2 \eta, 2 y).
$$
It follows that all properties which are valid for $M ^{p,q}_{(\omega )}$
carry over to $\splM ^{p,q}_{(\omega )}$.

\par

From
\begin{equation}\label{FourSTFTs}
V_{\widehat \phi}\widehat f (\xi ,-x) =e^{i\scal x\xi}V_{\phi}f(x,\xi )
\end{equation}
it follows that
$$
f\in W^{q,p}_{(\omega )}(\rr d)\quad \Leftrightarrow \quad
\widehat f\in M^{p,q}_{(\omega _0)}(\rr d),\qquad \omega _0(\xi
,-x)=\omega (x,\xi ).
$$
In the symplectic situation these formulas read
\begin{equation}\label{stftsymplfour}
\maclV _{\mascF _\sigma \Phi}(\mascF _\sigma a)(X,Y) =
e^{2i\sigma (Y,X)}\maclV _\Phi a(Y,X)
\end{equation}
and
\begin{equation}\label{twistfourmod}
\mascF _\sigma \splM ^{p,q}_{(\omega )}(\rr {2d}) = \splW
^{q,p}_{(\omega _0)}(\rr {2d}), \qquad \omega _0(X,Y)=\omega (Y,X).
\end{equation}

\par

For brevity we denote $\splM ^p _{(\omega )}= \splM ^{p,p}_{(\omega
)}$, $\splW ^p_{(\omega )}=\splW ^{p,p}_{(\omega
)}$, and when $\omega \equiv 1$ we write $\splM ^{p,q}=\splM
^{p,q}_{(\omega )}$ and $\splW ^{p,q}=\splW ^{p,q}_{(\omega )}$. We
also let $\splM ^{p,q} _{(\omega )} (\rr {2d})$ be the
completion of $\maclS _s(\rr {2d})$ with
respect to the norm $\nm \cdo {\splM ^{p,q}_{(\omega )}}$.

\par

In the following proposition we list some basic facts on invariance, growth
and duality for modulation spaces. For any $p\in (0,\infty]$, its conjugate
exponent $p'\in [1,\infty ]$ is defined by
$$
p'=
\begin{cases}
\infty , & p\in (0,1],
\\[1ex]
\displaystyle{\frac p{p-1}}, & p\in (1,\infty),
\\[2ex]
1, & p=\infty .
\end{cases}
$$ 
Since our main results are formulated in terms of
symplectic modulation spaces, we state the result for them
instead of the modulation spaces $M ^{p,q}_{(\omega )}(\rr {d})$.

\par

\begin{prop}\label{p1.4}
Let $p,q,p_j,q_j\in [1,\infty ]$ for $j=1,2$, and $\omega
,\omega _1,\omega _2,v\in \mascP _E (\rr {4d})$ be such that
$v=\check v$, $\omega$ is $v$-moderate and $\omega _2\lesssim
\omega _1$. Then the following is true:
\begin{enumerate}
\item[{\rm{(1)}}] $a\in \splM ^{p,q}_{(\omega )}(\rr {2d})$ if and only if
\eqref{modnorm1} holds for any $\phi \in \splM ^1_{(v)}(\rr {2d})\setminus
0$. Moreover, $\splM ^{p,q}_{(\omega )}(\rr {2d})$ is a
Banach space under the norm
in \eqref{modnorm1} and different choices of $\phi$ give rise to
equivalent norms;

\vrum

\item[{\rm{(2)}}] if  $p_1\le p_2$ and $q_1\le q_2$  then
$$
\Sigma _1 (\rr {2d})\subseteq \splM ^{p,q}_{(\omega )}(\rr
{2d})\subseteq \Sigma _1 '(\rr {2d})
\quad \text{and}\quad
\splM ^{p_1,q_1}_{(\omega _1)}(\rr
{2d})\subseteq \splM ^{p_2,q_2}_{(\omega _2)}(\rr {2d}).
$$
If in addition $v\in \mascP (\rr {2d})$, then
$$
\mascS (\rr {2d})\subseteq \splM ^{p,q}_{(\omega )}(\rr
{2d})\subseteq \mascS '(\rr {2d});
$$

\vrum

\item[{\rm{(3)}}] the $L^2$ inner product $( \cdo ,\cdo )_{L^2}$ on $\maclS _{1/2}$
extends uniquely to a continuous sesquilinear form $( \cdo ,\cdo )$
on $\splM ^{p,q}_{(\omega )}(\rr {2d})\times \splM ^{p'\! ,q'}_{(1/\omega )}(\rr {2d})$.
On the other hand, if $\nmm a = \sup \abp {(a,b)}$, where the supremum is
taken over all $b\in \maclS _{1/2} (\rr {2d})$ such that
$\nm b{\splM ^{p',q'}_{(1/\omega )}}\le 1$, then $\nmm {\cdot}$ and $\nm
\cdot {\splM ^{p,q}_{(\omega )}}$ are equivalent norms;

\vrum

\item[{\rm{(4)}}] if $p,q<\infty$, then $\maclS _{1/2} (\rr {2d})$ is dense in
$\splM ^{p,q}_{(\omega )}(\rr {2d})$ and the dual space of $\splM
^{p,q}_{(\omega )}(\rr {2d})$ can be identified
with $\splM ^{p'\! ,q'}_{(1/\omega )}(\rr {2d})$, through the form
$(\cdo  ,\cdo )$. Moreover, $\maclS _{1/2} (\rr {2d})$ is weakly dense
in $\splM ^{p' ,q'}_{(\omega )}(\rr {2d})$ with respect to the form $(\cdo  ,\cdo )$
provided $(p,q) \neq (\infty,1)$ and $(p,q) \neq (1,\infty)$;

\vrum

\item[{\rm{(5)}}] if $p,q,r,s,u,v \in [1,\infty]$, $0\le \theta \le 1$,
\begin{equation*}
\frac1p = \frac {1-\theta }{r}+\frac {\theta}{u} \quad
\text{and} \quad \frac 1q = \frac {1-\theta }{s}+\frac {\theta}{v},
\end{equation*}
then complex interpolation gives
\begin{equation*}
(\splM ^{r,s}_{(\omega )},\splM ^{u,v}_{(\omega )})_{[\theta ]}
= \splM ^{p,q}_{(\omega )}.
\end{equation*}
\end{enumerate}
Similar facts hold if the $\splM ^{p,q}_{(\omega )}$ spaces are replaced by
the $\splW ^{p,q}_{(\omega )}$ spaces.
\end{prop}

\par

The proof of Proposition \ref{p1.4}
can be found in \cite {Cordero,Feichtinger1,Feichtinger2,
Feichtinger3, Feichtinger4, Feichtinger5, Grochenig2, Toft2,
Toft4, Toft5, Toft8}.

\par

In fact, (1) follows from Gr{\"o}chenig's argument verbatim in
\cite[Proposition 11.3.2 (c)]{Grochenig2}. Note that the window
class $\splM ^1_{(v)}(\rr {2d})$ in (1) contains $\Sigma _1(\rr {2d})$,
which in turn contains $\maclS _{1/2}(\rr {2d})$. Furthermore, if in
addition $v\in \mascP (\rr {4d})$,
then $\splM ^1_{(v)}(\rr {2d})$ contains $\mascS (\rr {2d})$.

\par

The proof of (2) in \cite[Chapter 12]{Grochenig2} is based on Gabor
frames and formulated for polynomial type weights $\mascP
(\rr {4d})$. These arguments also hold for the broader weight class $\mascP
_E(\rr {4d})$. Another way to prove this is by means of
\cite[Lemma 11.3.3]{Grochenig2} and Young's inequality.

\par

The assertions (3)--(5) in Proposition \ref{p1.4} can be found for
more general weights in Theorem 4.17, and a combination of
Theorem 3.4 and Proposition 5.2  in \cite{Toft8}.

\par

\begin{rem}
For $p,q\in (0,\infty ]$ (instead of $p,q\in [1,\infty ]$ as in Proposition
\ref{p1.4}), $a\in \splM ^{p,q}_{(\omega )}(\rr {2d})$ if and only if
\eqref{modnorm1} holds for any $\phi \in \Sigma _1(\rr {2d})\setminus
0$. Moreover, $\splM ^{p,q}_{(\omega )}(\rr {2d})$ is a
quasi-Banach space under the quasi-norm
in \eqref{modnorm1} and different choices of $\phi$ give rise to
equivalent quasi-norms. (See e.{\,}g. \cite{GaSa,Toft13}.) 
\end{rem}

\par

\begin{rem}\label{remGSmodident}
Let $\mathcal P$ be the set of all
$\omega \in \mascP _E(\rr {4d})$ such that
$$
\omega (X,Y ) = e^{c(|X|^{1/s}+|Y|^{1/s})},
$$
for some $c>0$. (Note that this implies that $s\ge 1$.) Then
\begin{alignat*}{2}
\bigcap _{\omega \in \mathcal P}\splM ^{p,q}_{(\omega )}(\rr {2d}) &=
\Sigma _s(\rr {2d}),
&\quad \phantom{\text{and}}\quad
\bigcup _{\omega \in \mathcal P}\splM ^{p,q}_{(1/\omega )}(\rr {2d}) &=
\Sigma _s'(\rr {2d})
\\[1ex]
\bigcup _{\omega \in \mathcal P}\splM ^{p,q}_{(\omega )}(\rr {2d}) &=
\maclS _s(\rr {2d}),
&
\bigcap _{\omega \in \mathcal P}\splM ^{p,q}_{(1/\omega )}(\rr {2d}) &=
\maclS _s'(\rr {2d}),
\end{alignat*}
and for $\omega \in \mathcal P$
$$
\Sigma _s(\rr {2d})\subseteq \splM ^{p,q}_{(\omega )}(\rr {2d})
\subseteq
\maclS _s(\rr {2d}) \quad \text{and}\quad
\maclS _s'(\rr {2d}) \subseteq \splM ^{p,q}_{(1/\omega )}(\rr {2d}) \subseteq
\Sigma _s'(\rr {2d}).
$$
(Cf. \cite[Prop.~4.5]{CPRT10}, \cite[Prop.~4]{GZ}, \cite[Cor.~5.2]{Pilipovic1} and
\cite[Thm.~4.1]{Teo2}. See also \cite[Thm.~3.9]{Toft8} for an extension of these
inclusions to broader classes of Gelfand--Shilov and modulation spaces.)
\end{rem}

\medspace

By Proposition \ref{p1.4} (4) we have norm density of $\maclS _{1/2}$ in
$\splM ^{p,q}_{(\omega )}$ when $p,q<\infty$. We may relax the assumptions
on $p$, provided we replace the norm convergence with \emph{narrow
convergence}.
This concept, that allows us to approximate elements in $\splM ^{\infty,q}
_{(\omega )}(\rr {2d})$ for $1 \leq q < \infty$, 
is treated in \cite{Sjo1,Toft2,Toft4}, and, for
the current setup of possibly exponential weights, in \cite{Toft8}.
(Sj{\"o}strand's original definition in \cite{Sjo1} is somewhat different.)
Narrow convergence is defined by means of the function
$$
H_{a,\omega ,p}(Y) \equiv \| \maclV _\Phi a(\cdot,Y)\omega (\cdot,Y) \|
_{L^p(\rr {2d})}, \quad Y \in \rr {2d},
$$
for $a \in \maclS _{1/2}'(\rr {2d})$, $\omega \in \mascP _E(\rr {4d})$,
$\Phi \in \maclS _{1/2}(\rr {2d}) \setminus 0$ and $p\in [1,\infty]$.

\par

\begin{defn}\label{p2.1}
Let $p,q\in [1,\infty
]$, and $a,a_j\in
\splM ^{p,q}_{(\omega )}(\rr {2d})$, $j=1,2,\dots \ $. Then $a_j$
is said to \emph{converge narrowly} to $a$ with respect to $p,q$, $\Phi \in
\maclS _{1/2}(\rr {2d})\setminus 0$ and $\omega \in \mascP
_E(\rr {4d})$, if there exist $g_j,g \in L^q(\rr {2d})$ such that:

\begin{enumerate}
\item[{\rm{(1)}}] $a_j\to a$ in $\maclS _{1/2}'(\rr {2d})$ as $j\to \infty$;

\par

\item[{\rm{(2)}}] $H_{a_j,\omega ,p} \le g_j$ and $g_j \to g$ in $L^q(\rr
{2d})$ and a.{\,}e. as $j\to \infty$.
\end{enumerate}
\end{defn}

\par

\begin{prop}\label{narrowprop}
If $\omega \in \mascP _E(\rr {4d})$ and $1 \le q < \infty$ then the following
is true: 
\begin{enumerate}
\item[{\rm{(1)}}] $\maclS _{1/2}(\rr {2d})$ is
dense in $\splM ^{\infty,q}_{(\omega )}(\rr {2d})$ with respect to narrow
convergence;
\item[{\rm{(2)}}] $\splM ^{\infty,q}_{(\omega )}(\rr {2d})$ is sequentially
complete with respect to the topology defined by narrow convergence. 
\end{enumerate}
\end{prop}

\par

We refer to \cite{CoToWa} for a proof of Proposition
\ref{narrowprop}.

\par

\subsection{Pseudo-differential operators}

\par

Next we recall some basic facts from pseudo-differential calculus
(cf. \cite{Ho3}). Let $\GL (d,\Omega)$ be the set of
all $d\times d$ matrices with entries in $\Omega$,
$s\ge 1/2$, $a\in \maclS _s
(\rr {2d})$, and $A\in \GL (d,\mathbf R)$ be fixed. Then the
pseudo-differential operator $\op _A(a)$ defined by
\begin{equation}\label{e0.5}
\op _A(a)f(x) = (2\pi )^{-d}\iint _{\rr {2d}}a(x-A(x-y),\xi )
f(y)e^{i\scal {x-y}\xi}\, dyd\xi
\end{equation}
is a linear and continuous operator on $\maclS _s (\rr d)$.
For $a\in \maclS _s'(\rr {2d})$ the
pseudo-differential operator $\op _A(a)$ is defined as the continuous
operator from $\maclS _s(\rr d)$ to $\maclS _s'(\rr d)$ with
distribution kernel given by
\begin{equation}\label{atkernel}
K_{a,A}(x,y)=(2\pi )^{-\frac d2}(\mascF _2^{-1}a)(x-A(x-y),x-y).
\end{equation}
Here $\mascF _2F$ is the partial Fourier transform of $F(x,y)\in
\maclS _s'(\rr {2d})$ with respect to the variable $y \in \rr d$. This
definition generalizes \eqref{e0.5} and is well defined, since the mappings
\begin{equation}\label{homeoF2tmap}
\mascF _2\quad \text{and}\quad F(x,y)\mapsto F(x-A(x-y),y-x)
\end{equation}
are homeomorphisms on $\maclS _s'(\rr {2d})$.
The map $a\mapsto K_{a,A}$ is hence a homeomorphism on
$\maclS _s'(\rr {2d})$.

\par

If $A=0$, then $\op _A(a)$ is the standard or Kohn-Nirenberg representation
$a(x,D)$. If instead $A=\frac 12 I$, then $\op _A(a)$ agrees
with the Weyl operator or Weyl quantization $\op ^w(a)$.

\par

For any $K\in \maclS '_s(\rr {d_1+d_2})$, let $T_K$ be the
linear and continuous mapping from $\maclS _s(\rr {d_1})$
to $\maclS _s'(\rr {d_2})$ defined by
\begin{equation}\label{pre(A.1)}
(T_Kf,g)_{L^2(\rr {d_2})} = (K,g\otimes \overline f )_{L^2(\rr {d_1+d_2})},
\quad f \in \maclS _s(\rr {d_1}), \quad g \in \maclS _s(\rr {d_2}).
\end{equation}
It is a well known consequence of the Schwartz kernel
theorem that if $t\in \mathbf R$, then $K\mapsto T_K$ and $a\mapsto
\op _A(a)$ are bijective mappings from $\mascS '(\rr {2d})$
to the space of linear and continuous mappings from $\mascS (\rr d)$ to
$\mascS '(\rr d)$ (cf. e.{\,}g. \cite{Ho3}).

\par

Likewise the maps $K\mapsto T_K$
and $a\mapsto \op _A(a)$ are uniquely extendable to bijective
mappings from $\maclS _s'(\rr {2d})$ to the set of linear and
continuous mappings from $\maclS _s(\rr d)$ to $\maclS _s'(\rr d)$.
In fact, the asserted bijectivity for the map $K\mapsto T_K$ follows from
the kernel theorem \cite[Theorem 2.3]{LozPer} (cf. \cite[vol. IV]{GS}).
This kernel theorem corresponds to the Schwartz kernel theorem
in the usual distribution theory.
The other assertion follows from the fact that $a\mapsto K_{a,A}$
is a homeomorphism on $\maclS _s'(\rr {2d})$.

\par

In particular, for each $a_1\in \maclS _s '(\rr {2d})$ and $A_1,A_2\in
\GL (d,\mathbf R)$, there is a unique $a_2\in \maclS _s '(\rr {2d})$ such that
$\op _{A_1}(a_1) = \op _{A_2} (a_2)$. The relation between $a_1$ and $a_2$
is given by
\begin{equation}\label{calculitransform}
\op _{A_1}(a_1) = \op _{A_2}(a_2) \quad \Leftrightarrow \quad
a_2(x,\xi )=e^{i\scal {(A_1-A_2)D_\xi }{D_x}}a_1(x,\xi ).
\end{equation}
(Cf. \cite{Ho3}.) Note that the right-hand side makes sense, since
it means $\widehat a_2(\xi ,x)=e^{i\scal {(A_1-A_2)x}{\xi}}\widehat a_1(\xi ,x)$,
and since the map $a(\xi ,x)\mapsto e^{i\scal {Ax}\xi }a(\xi ,x)$ is continuous on
$\maclS _s '(\rr {2d})$.

\par

For future references we observe the relationship
\begin{equation}\label{Eq:STFTLinkKernelSymbol}
\begin{aligned}
|(V_\Phi &K_{a,A})(x,y,\xi ,-\eta )|
\\
&=
|(V_\Psi a)(x-A(x-y),A^*\xi +(I-A^*)\eta ,\xi -\eta ,y-x)|,
\\[1ex]
\Phi (x,y) &= (\mascF _2\Psi )(x-A(x-y),x-y)
\end{aligned}
\end{equation}
between symbols and kernels for pseudo-differential operators,
which follows by straight-forward applications of Fourier inversion
formula (see also the proof of Proposition 2.5 in \cite{Toft15}).
We observe that for the Weyl case, \eqref{Eq:STFTLinkKernelSymbol}
takes the convenient form
\begin{equation}\label{Eq:STFTLinkKernelSymbolWeyl}
\begin{aligned}
|(V_\Phi K_{a}^w)(x,y,\xi ,-\eta )|
&=
\left | (\maclV _\Psi a)\left ({\textstyle{\frac 12}} (Y+X) ,{\textstyle{\frac 12}}(Y-X) \right ) \right |,
\\[1ex]
\Phi (x,y) &= (\mascF _2\Psi )(\textstyle{\frac 12}(x+y),x-y),
\\[1ex]
X&=(x,\xi )\in \rr {2d},\quad Y=(y,\eta )\in \rr {2d}.
\end{aligned}
\end{equation}
when using symplectic short-time Fourier transforms. Here
$K_{a}^w$ is the kernel of $\op ^w(a)$.

\medspace

Next we discuss symbol products in pseudo-differential calculi,
twisted convolution and related
operations (see \cite{Ho3,Folland}). Let $A\in \GL (d,\mathbf R)$,
$s\ge 1/2$ and let $a,b\in
\maclS _s '(\rr {2d})$. The pseudo product with respect to
$A$ or the \emph{$A$-pseudo product} $a\wpr _{\! A}
b$ between $a$
and $b$ is the function or distribution which satisfies
\begin{align*}
\op _A(a\wpr _{\! A} b) &= \op _A(a)\circ \op _A(b),
\intertext{provided the right-hand side
makes sense as a continuous operator from $\maclS _s(\rr d)$ to
$\maclS _s'(\rr d)$. Since the Weyl case is especially important,
we put $a\wpr b =a\wpr _{\! A}b$ when $A=\frac 12 I_d$, where
$I_d$ is the unit matrix of order $d$. That is, we have}
\op ^w(a\wpr b) &= \op ^w(a)\circ \op ^w(b),
\end{align*}
provided the right-hand side makes sense.
%
%
%
%
%

\par

\subsection{The Weyl product and the twisted convolution}

\par

The symplectic Fourier transform ${\mascF _\sigma}$ is continuous on
$\maclS _s (\rr {2d})$ and extends uniquely to a homeomorphism on
$\maclS _s '(\rr {2d})$, and to a unitary map on $L^2(\rr {2d})$, since similar
facts hold for $\mascF$. Furthermore $\mascF _\sigma^{2}$ is the identity
operator.

\par

Let $s\ge 1/2$ and $a,b\in \maclS _s (\rr {2d})$. The \emph{twisted
convolution} of $a$ and $b$ is defined by
\begin{equation}\label{twist1}
(a \ast _\sigma b) (X)
= (2/\pi)^{\frac d2} \int _{\rr {2d}}a(X-Y) b(Y) e^{2 i \sigma(X,Y)}\, dY.
\end{equation}
The definition of $*_\sigma$ extends in different ways. For example
it extends to a continuous multiplication on $L^p(\rr {2d})$ when $p\in
[1,2]$, and to a continuous map from $\maclS _s '(\rr {2d})\times
\maclS _s (\rr {2d})$ to $\maclS _s '(\rr {2d})$. If $a,b \in
\maclS _s '(\rr {2d})$, then $a \wpr b$ makes sense if and only if $a
*_\sigma \widehat b$ makes sense, and
\begin{equation}\label{tvist1}
a \wpr b = (2\pi)^{-\frac d2} a \ast_\sigma (\mascF _\sigma {b}).
\end{equation}
For the twisted convolution we have
\begin{equation}\label{weylfourier1}
\mascF _\sigma (a *_\sigma b) = (\mascF _\sigma a) *_\sigma b =
\check{a} *_\sigma (\mascF _\sigma b),
\end{equation}
where $\check{a}(X)=a(-X)$ (cf. \cite{Toft1}). A
combination of \eqref{tvist1} and \eqref{weylfourier1} gives
\begin{equation}\label{weyltwist2}
\mascF _\sigma (a\wpr b) = (2\pi )^{-\frac d2}(\mascF _\sigma
a)*_\sigma (\mascF _\sigma b).
\end{equation}
If $\widetilde a(X) = \overline{a(-X)}$ then
\begin{equation*}
(a_1*_\sigma a_2,b) = (a_1,b*_\sigma \widetilde a_2)=(a_2,\widetilde
a_1*_\sigma b),\quad (a_1*_\sigma a_2)*_\sigma b = a_1*_\sigma
(a_2*_\sigma b),
\end{equation*}
for appropriate $a_1, a_2,b$, and furthermore (cf. \cite{HTW})
\begin{equation}\label{duality0}
(a_1 \wpr a_2, b) = (a_2, \overline{a_1} \wpr b) = (a_1, b \wpr
\overline{a_2}).
\end{equation}

\par

We have the following result for the map $e^{it\scal {AD_\xi}{D_x}}$ in
\eqref{calculitransform} when the domains are modulation spaces. We refer
to \cite[Proposition 1.7]{Toft5} for the proof (see also
\cite[Proposition 6.14]{Toft8}).

\par

\begin{prop}\label{propCalculiTransfMod}
Let $\omega _0\in \mascP _E(\rr
{4d})$, $p,q\in [1,\infty ]$, $A_1,A_2\in \GL (d,\mathbf R)$, and set
$$
\omega _A(x,\xi ,\eta ,y)= \omega _0(x-Ay,\xi -A^*\eta ,\eta ,y).
$$
The map $e^{i\scal {AD_\xi}{D_x}}$ on $\maclS _{1/2}'(\rr {2d})$
restricts to a homeomorphism from $M^{p,q}_{(\omega
_0)}(\rr {2d})$ to $M^{p,q}_{(\omega _A)}(\rr {2d})$.

\par

In particular, if $a_1,a_2\in \maclS _{1/2}'(\rr {2d})$ satisfy
\eqref{calculitransform}, then $a_1\in
M^{p,q}_{(\omega _{A_1})}(\rr {2d})$, if and only if $a_2\in
M^{p,q}_{(\omega _{A_2})}(\rr {2d})$.
\end{prop}

\par

(Note that
in the equality of (2) in \cite[Proposition 6.14]{Toft8},
$y$ and $\eta$ should be interchanged in the last two arguments
in $\omega _0$.)

\par

\section{Continuity for the Weyl product on
modulation spaces}\label{sec2}

\par

In this section we deduce results on sufficient conditions for
continuity of the Weyl product on
modulation spaces, and the twisted convolution on Wiener amalgam spaces. 
The main results are Theorem
\ref{Thm:MainThmOdd} and \ref{Thm:MainThmOdd2}
concerning the Weyl product and more general products in
pseudo-differential calculus,
and Theorem \ref{Thm:MainThmOddTwistConv} concerning the twisted
convolution.

\par

When proving Theorem \ref{Thm:MainThmOdd} we first need norm estimates. 
Then we prove the uniqueness of the extension, where generally norm
approximation not suffices, since the test function space may fail
to be dense in several of the domain spaces. The situation is
saved by a comprehensive argument based on narrow convergence.
First we prove the important
special cases Propositions \ref{Prop:Prop1} and \ref{Prop:Prop2} and then
we deduce Theorem \ref{Thm:MainThmOdd}.

\par

For $N \ge 2$ we let $\masfR _N$ be the
function on $[0,1]^{N+1}$, given by
\begin{equation}\label{Eq:HYfunctional}
\begin{aligned}
\masfR _N(x) &= ({N-1})^{-1}\left ({\sum _{j=0}^N
x_j-1}\right ),
\\[1ex]
x &= (x_0,x_1,\dots ,x_N)\in [0,1]^{N+1}, 
\end{aligned}
\end{equation}
and we consider mappings of the form
\begin{align}
(a_1,\dots ,a_N)
&\mapsto
a_1\wpr \cdots \wpr a_N,
\label{Eq:Weylmap}
\intertext{or, more generally, mappings of the form}
(a_1,\dots ,a_N)
&\mapsto
a_1\wpr _{\! A}\cdots \wpr _{\! A} a_N,
\tag*{(\ref{Eq:Weylmap})$'$}
\end{align}
We observe that
\begin{equation}\label{Eq:HYfunctionalConj}
\masfR _N({\textstyle{\frac 1p}}) + \masfR _N({\textstyle{\frac 1{p'}}}) = 1.
\end{equation}

\par

We first show a formula for the STFT
of $a_1\wpr \cdots \wpr a_N$ expressed with
\begin{equation}\label{Fjdef}
F_j(X,Y) = \maclV_{\Phi _j}a_j (X+Y,X-Y).
\end{equation}

\par

The following lemma is a restatement of \cite[Lemma 2.3]{CoToWa}.
The proof is therefore omitted.

\par

\begin{lemma}\label{prodlemma}
Let $\Phi _j \in \maclS _{1/2}(\rr {2d})$, $j=1,\dots ,N$, $a_k \in
\maclS _{1/2}'(\rr {2d})$ for some $1 \le k \le N$, and $a_j \in
\maclS _{1/2}(\rr {2d})$ for $j \in \{1,\dots ,N\} \setminus k$. 
Suppose
$$
\Phi _0 = \pi ^{(N-1)d}\Phi _1\wpr \cdots \wpr \Phi _N\quad
\text{and}\quad
a_0 = a_1\wpr \cdots \wpr a_N.
$$
If $F_j$ are given by \eqref{Fjdef} then
\begin{multline}\label{STFTintegral}
F_0(X_N,X_0)
\\[1ex]
=\idotsint _{\rr {2(N-1)d}}e^{2 i Q(X_0,\dots  ,X_N)}
\prod _{j=1}^NF_j(X_j,X_{j-1}) \, dX_1
\cdots dX_{N-1}
\end{multline}
with 
$$
Q(X_0,\dots,X_N)=\sum_{j=1}^{N-1}\sigma(X_j-X_0,X_{j+1}-X_0).
$$ 
\end{lemma}

\par

Next we use the previous lemma to find sufficient conditions
for the extension of \eqref{Eq:Weylmap} to modulation spaces.
The integral representation of $V_{\Phi _0}a_0$ in the previous
lemma leads to the weight condition
\begin{multline}\label{Eq:WeightCond}
1 \lesssim \omega _0(X_N+X_0,X_N-X_0)\prod _{j=1}^N
\omega _j(X_j+X_{j-1},X_j-X_{j-1}),
\\[1ex]
X_0,X_1,\dots ,X_N\in \rr {2d}.
\end{multline}

\par

The following result
is a restatement of \cite[Proposition 2.2]{CoToWa}. The proof is
therefore omitted.

\par

\begin{prop}\label{Prop:Prop1}
Let $p_j,q_j\in [1,\infty ]$, $j=0,1,\dots , N$, and suppose
$$
\masfR _N({\textstyle{\frac 1{q'}}})\le 0\le \masfR _N({\textstyle{\frac 1p}}).
$$
Let $\omega _j$, $j=0,1,\dots ,N$, and suppose
\eqref{Eq:WeightCond} holds. Then the map \eqref{Eq:Weylmap}
from $\maclS _{1/2}(\rr {2d}) \times \cdots \times
\maclS _{1/2}(\rr {2d})$ to $\maclS _{1/2}(\rr {2d})$
extends uniquely to a continuous and
associative map from $\splM ^{p_1,q_1}_{(\omega _1)}(\rr {2d})
\times \cdots \times
\splM ^{p_N,q_N}_{(\omega _N)}(\rr {2d})$ to $\splM ^{p_0',q_0'}
_{(1/\omega _0)}(\rr {2d})$.
\end{prop}

\par

The associativity means that for any
product \eqref{Eq:Weylmap}, where the factors $a_j$ satisfy
the hypotheses, the subproduct
$$
a_{k_1}\wpr a_{k_1+1} \wpr  \cdots \wpr a_{k_2}
$$
is well defined as a distribution for any $1\le k_1 \le k_2\le N$, and
$$
a_1\wpr \cdots \wpr a_N = (a_1\wpr \cdots \wpr a_k)
\wpr (a_{k+1}\wpr \cdots \wpr a_N),
$$
for any $1\le k\le N-1$.
 
\par

For appropriate weights $\omega$ the space
$\splM ^2_{(\omega )}(\rr {2d})$ consists of symbols of Hilbert--Schmidt
operators acting between certain modulation spaces (cf. \cite{Toft5,Toft9}). 

\par

The next result is an extension of this fact.

\par

\begin{prop}\label{Prop:Prop2}
Let $N\ge 3$ be odd,
$p,p_j\in (0,\infty ]$, $j=1,\dots , N$,
and let
$\omega _j \in \mascP _E(\rr {4d})$, $j=0,1,\dots ,N$, and
suppose \eqref{Eq:WeightCond} holds. 
Then the following is true:
\begin{enumerate}
\item[{\rm{(1)}}] if $p_0=p_N=p$, $p_j=\max (1,p)$ when $j\in [3,N-2]$ is odd
and $p_j=p'$ when $j$ is even, then the map \eqref{Eq:Weylmap} 
from $\maclS _{1/2}(\rr {2d}) \times \cdots \times
\maclS _{1/2}(\rr {2d})$ to $\maclS _{1/2}(\rr {2d})$
extends uniquely to a continuous and associative map from
$\splM ^{p_1} _{(\omega _1)}(\rr {2d}) \times \cdots
\times \splM ^{p_N} _{(\omega _N)}(\rr {2d})$ to
$\splM ^{p} _{(1/\omega _0)}(\rr {2d})$;

\par

\item[{\rm{(1)}}] if $p_j=p$ when $j$ is even and
$p_j=p'$ when is odd,
then the map \eqref{Eq:Weylmap} 
from $\maclS _{1/2}(\rr {2d}) \times \cdots \times
\maclS _{1/2}(\rr {2d})$ to $\maclS _{1/2}(\rr {2d})$
extends uniquely to a continuous and associative map from
$\splM ^{p_1} _{(\omega _1)}(\rr {2d}) \times \cdots
\times \splM ^{p_N} _{(\omega _N)}(\rr {2d})$ to
$\splM ^{p'} _{(1/\omega _0)}(\rr {2d})$.
\end{enumerate}
\end{prop}

\par

Proposition \ref{Prop:Prop2} follows by combining
\eqref{Eq:STFTLinkKernelSymbolWeyl} with the following
result for kernel operators. The details are left for the reader.

\par

\begin{prop}\label{Prop:Prop2Kernels}
Let $N\ge 3$ be odd,
$p,p_j\in (0,\infty ]$, $j=1,\dots , N$,
and let $\omega _j \in \mascP _E(\rr {4d})$,
$j=0,1,\dots ,N$, be such that
\begin{equation}\label{Eq:KernelAlgWeightCond}
\inf _{x_j,\xi _j\in \rr d}
\left (
\omega _0(x_0,x_N,\xi _0,-\xi _N) \prod _{j=1}^N
\omega _j(x_{j-1},x_j,\xi _{j-1},-\xi _j)
\right ) >0.
\end{equation}
Then the following is true:
\begin{enumerate}
\item[{\rm{(1)}}] if $p_0=p_N=p$, $p_j=\max (1,p)$ when $j\in [3,N-2]$ is odd
and $p_j=p'$ when $j$ is even, then
\begin{equation}\label{Eq:MultLinKernelMap}
(K_1,K_2,\dots ,K_N) \mapsto K_1\circ K_2\circ \cdots \circ K_N
\end{equation}
from $\maclS _{1/2}(\rr {2d}) \times \cdots \times
\maclS _{1/2}(\rr {2d})$ to $\maclS _{1/2}(\rr {2d})$
extends uniquely to a continuous and associative map from
$M ^{p_1} _{(\omega _1)}(\rr {2d}) \times \cdots
\times M ^{p_N} _{(\omega _N)}(\rr {2d})$ to
$M ^{p} _{(1/\omega _0)}(\rr {2d})$, and
\begin{equation}\label{Eq:MultLinKernelMapEst}
\begin{aligned}
\nm {K_1\circ K_2\circ \cdots \circ K_N}{M^p_{(1/\omega _0)}}
&\lesssim
\prod _{j=1}^N \nm {K_j}{M^{p_j}_{(\omega _j)}},
\\[1ex]
K_j &\in M^{p_j}_{(\omega _j)}(\rr {2d}),\ j=1,\dots ,N\text ;
\end{aligned}
\end{equation}

\par

\item[{\rm{(2)}}] if $p_j=p$ when $j$ is even and
$p_j=p'$ when is odd,
then the map \eqref{Eq:MultLinKernelMap}
from $\maclS _{1/2}(\rr {2d}) \times \cdots \times
\maclS _{1/2}(\rr {2d})$ to $\maclS _{1/2}(\rr {2d})$
extends uniquely to a continuous and associative map from
$M ^{p_1} _{(\omega _1)}(\rr {2d}) \times \cdots
\times M ^{p_N} _{(\omega _N)}(\rr {2d})$ to
$M ^{p'} _{(1/\omega _0)}(\rr {2d})$.
\end{enumerate}
\end{prop}

\par

\begin{proof}
We observe that \eqref{Eq:KernelAlgWeightCond} is the same
as
$$
\inf _{x_j,\xi _j\in \rr d}
\left (
\omega _0(x_0,x_N,\xi _0,\xi _N) \prod _{j=1}^N
\omega _j(x_{j-1},x_j,(-1)^{j-1}\xi _{j-1},(-1)^{j-1}\xi _j)
\right ) >0.
$$
First suppose that $K_j\in \maclS _{1/2}(\rr {2d})$ for every
$j$.
Let $p_0=\max (p,1)$,
\begin{align*}
\widetilde K_j
&=
\begin{cases}
K_j,& j\ \text{odd}
\\[1ex]
\overline{K_j},& j\ \text{even}
\end{cases}
\\[1ex]
\widetilde \omega _j(x,y,\xi ,\eta )
&=
\begin{cases}
\omega _j(x,y,\xi ,\eta ),& j\ \text{odd}
\\[1ex]
\omega _j(x,y,-\xi ,-\eta ),& j\ \text{even}
\end{cases}
\\[1ex]
y &= (x_2,x_3,\dots ,x_{N-2}),
\qquad
\eta = (\xi _2,\xi _3,\dots ,\xi _{N-2}),
\\[1ex]
G(x_0,x_N,x_1,x_{N-1}) &= \widetilde K_1(x_0,x_1)
\widetilde K_N(x_{N-1},x_N),
\\[1ex]
H_1 (x_1,x_{N-1},y)
&=
\prod _{j=1}^{(N-1)/2}\widetilde K_{2j}(x_{2j-1},x_{2j}),
\\[1ex]
H_2 (y) &=
\prod _{j=1}^{(N-3)/2}\widetilde K_{2j+1}(x_{2j},x_{2j+1}),
\\[1ex]
H(x_0,x_N) &= (H_2,H_1(x_1,x_{N-1},\cdo ))_{L^2},
\end{align*}
\begin{align*}
\vartheta _0(x_0,x_N,x_1,x_{N-1},&\xi _0,\xi _N,\xi _1,\xi _{N-1})
\\
&=
\omega _1(x_0,x_1,\xi _0,\xi _1)\omega _N(x_{N-1},x_N,\xi _{N-1},\xi _N),
\\[1ex]
\vartheta _1(x_1,x_{N-1},y,\xi _1,\xi _{N-1},\eta )
&=
\prod _{j=1}^{(N-1)/2}\widetilde \omega _{2j}(x_{2j-1},x_{2j},\xi _{2j-1},\xi _{2j}),
\intertext{and}
\vartheta _2(y,\eta )
&=
\prod _{j=1}^{(N-3)/2}\widetilde \omega _{2j+1}(x_{2j},x_{2j+1},\xi _{2j},\xi _{2j+1}).
\end{align*}
Then it follows by straight-forward computations that
\begin{equation}\label{Eq:CompKernelRef}
(K_1\circ \cdots \circ K_N)(x_0,x_N) = (G(x_0,x_N,\cdo ),\overline H)_{L^2},
\end{equation}
and that the right-hand side makes sense as an element in $\maclS _{1/2}(\rr {2d})$
when $K_j\in \maclS _{1/2}'(\rr {2d})$ for even $j$. In the same way,
the right-hand side of \eqref{Eq:CompKernelRef}
makes sense as an element in $\maclS _{1/2}'(\rr {2d})$
when $K_j\in \maclS _{1/2}'(\rr {2d})$ for odd $j$. Hence the map
\eqref{Eq:MultLinKernelMap} extends to continuous mappings from
$$
\maclS _{1/2}(\rr {2d})\times \maclS _{1/2}'(\rr {2d}) \times \maclS _{1/2}(\rr {2d})
\times \cdots \times \maclS _{1/2}(\rr {2d})
$$
to $\maclS _{1/2}(\rr {2d})$, and from
$$
\maclS _{1/2}'(\rr {2d})\times \maclS _{1/2}(\rr {2d}) \times \maclS _{1/2}'(\rr {2d})
\times \cdots \times \maclS _{1/2}'(\rr {2d})
$$
to $\maclS _{1/2}'(\rr {2d})$. The uniquenesses of these extensions follows
by approximating those $K_j$ which belong to
$\maclS _{1/2}'(\rr {2d})$, by taking sequences
of elements $\maclS _{1/2}(\rr {2d})$ which converge to those $K_j$ in
$\maclS _{1/2}'(\rr {2d})$.

\par

We have that $\maclS _{1/2}(\rr {2d})$ is dense in $M^p_{(\omega _j)}(\rr {2d})
\subseteq \maclS _{1/2}'(\rr {2d})$ when
$p<\infty$, and that $p'=1<\infty$ when $p=\infty$
and $p\le 1<\infty$ when $p'=\infty$. Hence it follows from the recent
uniqueness properties, that the result follows if we prove that
\eqref{Eq:MultLinKernelMapEst} holds when $K_j\in \maclS _{1/2}(\rr {2d})$
for every $j$.

\par

We have
\begin{align}
\nm {\widetilde K_j}{M^{p_j}_{(\widetilde \omega _j)}}
&=
\nm {K_j}{M^{p_j}_{(\omega _j)}},
\\[1ex]
\nm G{M^p_{(\vartheta _0)}}
&\lesssim
\nm {K_1}{M^p_{(\omega _1)}}\nm {K_N}{M^p_{(\omega _N)}},
\label{Eq:KernelGNormEst} 
\\[1ex]
\nm {H_1}{M^{p'}_{(\vartheta _1)}(\rr {(N-1)d})}
&=
\prod _{j=1}^{(N-1)/2}
\nm {\widetilde K_{2j}}{M^{p'}_{(\widetilde \omega _{2j})}(\rr {2d})},
\notag
\\[1ex]
\nm {H_2}{M^{p'}_{(\vartheta _2)}(\rr {(N-3)d})}
&=
\prod _{j=1}^{(N-3)/2}
\nm {\widetilde K_{2j+1}}{M^{p'}_{(\widetilde \omega _{2j+1})}(\rr {2d})},
\notag
\intertext{which implies}
\nm {\overline H}{M^{p'}_{(\vartheta )}}
&\lesssim
\prod _{j=2}^{N-1}\nm {\widetilde K_j}{M^{p_j}_{(\widetilde \omega _j)}}
=
\prod _{j=2}^{N-1}\nm {K_j}{M^{p_j}_{(\omega _j)}}
\label{Eq:ModNormConjHEst}
\end{align}
A combination of this estimate with \eqref{Eq:CompKernelRef},
\eqref{Eq:KernelGNormEst} and \eqref{Eq:ModNormConjHEst}
gives \eqref{Eq:MultLinKernelMapEst}, and (1) follows.

\par

The assertion (2) follows by similar arguments and is left for the reader.
\end{proof}

\par

The following result now follows by interpolation between
Propositions \ref{Prop:Prop1} and \ref{Prop:Prop2}.
Here and in what follows we let
$$
I_N=[0,N]\cap \mathbf Z,
\quad \text{and}\quad 
\Omega _N =\sets {(j,k)\in I_N^2}{j+k\in 2\mathbf Z +1},
$$
and
\begin{equation}\label{Eq:sfQfunctional1}
\begin{aligned}
\masfQ _{0,N}(x,y)
&=
\min _{j+k\in 2\mathbf Z+1}
\left ( \frac {x_j+y_k}2 \right ),
\\[1ex]
\masfQ _N(x,y)
&=
\min _{j+k\in 2\mathbf Z+1}
\left ( \frac {x_j+y_k}2,1-\frac {x_j+y_k}2 \right ),
\\[1ex]
\masfQ _{0,N}(x) &= \masfQ _{0,N}(x,x),
\qquad
\masfQ _N(x) = \masfQ _N(x,x),
\\[1ex]
x
&=
(x_0,x_1,\dots ,x_N)\in [0,1]^{N+1},
\\[1ex]
y
&=
(y_0,y_1,\dots ,y_N)\in [0,1]^{N+1}.
\end{aligned}
\end{equation}

\par

\begin{prop}\label{Prop:MainPropOdd}
Let $N\ge 3$ be odd, $\masfR _N$ be as in \eqref{Eq:HYfunctional},
$\masfQ _N$ be as in \eqref{Eq:sfQfunctional1}, and let
$p_j,q_j\in [1,\infty ]$, $j=0,1,\dots , N$, be such that
\begin{equation}\label{Eq:pqConditionsA}
\max \left ( \masfR _N({\textstyle{\frac 1{q'}}}) ,0 \right )
\le  \min 
\left ( \masfQ _N({\textstyle{\frac 1p}}), \masfQ _N({\textstyle{\frac 1{q'}}}),
\masfQ _N({\textstyle{\frac 1p,\frac 1q}}),
\masfR _N({\textstyle{\frac 1p}})\right ).
\end{equation}
Also let $\omega _j \in \mascP _E(\rr {4d})$, $j=0,1,\dots ,N$, and
suppose \eqref{Eq:WeightCond} holds. 
Then the map \eqref{Eq:Weylmap} 
from $\maclS _{1/2}(\rr {2d}) \times \cdots \times
\maclS _{1/2}(\rr {2d})$ to $\maclS _{1/2}(\rr {2d})$
extends uniquely to a continuous and associative map from
$\splM ^{p_1,q_1} _{(\omega _1)}(\rr {2d}) \times \cdots
\times \splM ^{p_N,q_N} _{(\omega _N)}(\rr {2d})$ to
$\splM ^{p_0',q_0'} _{(1/\omega _0)}(\rr {2d})$.
\end{prop}

\par

We observe that $\masfQ _N({\textstyle{\frac 1{q'}}})
=\masfQ _N({\textstyle{\frac 1q}})$ when $q$ is the same as
in the previous proposition.

\par

\begin{proof}
Evidently, the result holds true when $\masfR _N(q ') \le 0$ in view
of Proposition \ref{Prop:Prop1}. We need to prove the result when
$\masfR _N({\textstyle{\frac 1{q'}}}) \ge 0$.

\par

We use the same notations as in Lemma \ref{Lemma:Athm0.3}
and its proof. By Propositions \ref{Prop:Prop1} and \ref{Prop:Prop2}
we have
\begin{align}
\splM ^{r_1,s_1}_{(\omega _1)}
\times \cdots \times
\splM ^{r_N,s_N}_{(\omega _N)} &\hookrightarrow \splM ^{r_0',s_0'}
_{(1/\omega _0)}
\label{Eq:WeylProdExprs}
\intertext{and}
\splM ^{u_1,u_1}_{(\omega _1)}
\times \cdots \times
\splM ^{u_N,u_N}_{(\omega _N)} &\hookrightarrow \splM ^{u_0',u_0'}
\label{Eq:WeylProdExpu}
\end{align}
when $r_j,s_j,u_j\in [1,\infty ]$, $j\in I_N$, satisfy
\begin{align}
\sum _{j=0}^N\frac 1{s_j'} &\le 1\le \sum _{j=0}^N\frac 1{r_j},
\label{Eq:reCond}
\intertext{and}
u_j &=
\begin{cases}
v', & j\in 2\mathbf Z,
\\[1ex]
v, & j\in 2\mathbf Z+1,
\end{cases}
\label{Eq:uTovCond}
\end{align}
for some $v\in [1,\infty ]$. Hence, by
combining Proposition \ref{p1.4} (5)
with multi-linear interpolation in \cite[Chapter 4]{BeLo},
we get
$$
\splM ^{p_1,q_1}_{(\omega _1)}
\times \cdots \times
\splM ^{p_N,q_N}_{(\omega _N)} \hookrightarrow \splM ^{p_0',q_0'}
_{(1/\omega _0)}
$$
when
\begin{alignat}{3}
\frac 1{p_j} &= \frac {1-\theta}{r_j}+\frac \theta{v'}, &
\qquad
\frac 1{q_j} &= \frac {1-\theta}{s_j}+\frac \theta{v'}, &
\qquad j&\in 2\mathbf Z
\label{Eq:InterpolCond1}
\intertext{and}
\frac 1{p_k} &= \frac {1-\theta}{r_k}+\frac \theta{v}, &
\qquad
\frac 1{q_k} &= \frac {1-\theta}{s_k}+\frac \theta{v}, &
\qquad k&\in 2\mathbf Z+1.
\label{Eq:InterpolCond2}
\end{alignat}
This gives
\begin{multline*}
\sum _{j=0}^N \frac 1{p_j} = 
(1-\theta )\sum _{j=0}^N \frac 1{r_j}
+ \theta \cdot \frac {N+1}2
\left (
\frac 1v+\frac 1{v'}
\right )
\\[1ex]
=
(1-\theta )\sum _{j=0}^N \frac 1{r_j}
+ \theta \cdot \frac {N+1}2 \ge 1+\theta \cdot \frac {N-1}2
\end{multline*}
and
\begin{multline*}
\sum _{j=0}^N \frac 1{q_j'} = 
(1-\theta )\sum _{j=0}^N \frac 1{s_j'}
+ \theta \cdot \frac {N+1}2
\left (
\frac 1v+\frac 1{v'}
\right )
\\[1ex]
=
(1-\theta )\sum _{j=0}^N \frac 1{s_j'}
+ \theta \cdot \frac {N+1}2 
\le 1+\theta \cdot \frac {N-1}2 ,
\end{multline*}
which implies
\begin{equation}\label{Eq:pqConditionsAStep1}
\masfR _N({\textstyle{\frac 1{q'}}})\le \frac \theta 2
\le \masfR _N({\textstyle{\frac 1p}}).
\end{equation}
In particular we have $R_N({\textstyle{\frac 1{q'}}}) \le \frac \theta 2$.

\par

By \eqref{Eq:InterpolCond1} and \eqref{Eq:InterpolCond2} we also get
$$
\frac 1{p_j}+\frac 1{p_k}
= \frac {1-\theta}{r_j}+\frac {1-\theta}{r_k}+\frac \theta{v}+\frac \theta{v'}
\ge
\frac \theta{v}+\frac \theta{v'}=\theta ,
$$
when $j+k$ is odd. That is,
\begin{equation}\label{Eq:IneqLebExp1}
\frac 12\left ( \frac 1{p_j}+\frac 1{p_k} \right ) \ge \frac \theta 2
\ge R_N({\textstyle{\frac 1{q'}}}).
\end{equation}
In the same way we get
$$
\frac 12\left ( \frac 1{q_j}+\frac 1{q_k} \right ) \ge \frac \theta 2
\ge R_N({\textstyle{\frac 1{q'}}})
\quad \text{and}\quad
\frac 12\left ( \frac 1{p_j}+\frac 1{q_k} \right ) \ge \frac \theta 2
\ge R_N({\textstyle{\frac 1{q'}}})
$$
when $j+k$ is odd. We also have
$$
\frac 1{p_j'}+\frac 1{p_k'}
= \frac {1-\theta}{r_j'}+\frac {1-\theta}{r_k'}+\frac \theta{v'}+\frac \theta{v}
\ge
\frac \theta{p}+\frac \theta{p'}=\theta ,
$$
when $j+k$ is odd, and it follows that \eqref{Eq:IneqLebExp1}
and its two following inequalities hold true with $p_j'$, $p_k'$,
$q_j'$ and $q_k'$ in place of $p_j$, $p_k$,
$q_j$ and $q_k$, respectively, at each occurrence. By combining
these inequality we get \eqref{Eq:pqConditionsA}.

\par

In order for verify the the interpolation completely, we need to prove
that if $p,q\in [1,\infty ]^{N+1}$ satisfy \eqref{Eq:pqConditionsA},
then there are $r,s,u\in [1,\infty ]^{N+1}$, $v\in [1,\infty ]$
and $\theta \in [0,1]$ such that
\eqref{Eq:WeylProdExprs}--\eqref{Eq:InterpolCond2} hold.
As remarked above, the result holds true if
$\masfR _N({\textstyle{\frac 1{q'}}})\le 0$. Therefore assume that
$\masfR _N({\textstyle{\frac 1{q'}}})>0$. By \eqref{Eq:pqConditionsA}
it follows that $\masfR _N({\textstyle{\frac 1{q'}}})\le \frac 12$. Choose
$\theta \in (0,1]$ such that $\masfR _N({\textstyle{\frac 1{q'}}})=\frac \theta 2$.
By reasons of symmetry we may assume that
$$
p_0=\min (j\in I_N)(p_j,q_j),
$$
and we shall consider the two cases when $\frac 1{p_0}\ge \frac \theta 2$
and when $\frac 1{p_0}< \frac \theta 2$, separately.

\par

First suppose that $\frac 1{p_0}\ge \frac \theta 2$. Then
$$
\min \left ( \frac 1{p},\frac 1{p'},\frac 1{q},\frac 1{q'},
\masfR _N({\textstyle{\frac 1p}})\right )
\ge \masfR _N({\textstyle{\frac 1{q'}}}),
$$
and the result follows from \cite[Proposition 2.5]{CoToWa}.

\par

Therefore assume that $\frac 1{p_0}< \frac \theta 2$ and let $v>2$ be
chosen such that
$$
\frac 1{p_0}=\frac \theta v.
$$
Then
\begin{equation}\label{Eq:LebExpEstEven}
\frac 1{p_j},\frac 1{q_j} \ge \frac 1{p_0} = \frac \theta v,
\qquad j\in 2\mathbf Z
\end{equation}
and since
$$
\masfQ _N({\textstyle{\frac 1p}}), \masfQ _N({\textstyle{\frac 1q}},
\masfQ _N({\textstyle{\frac 1p,\frac 1q}})
\ge
\masfR _N({\textstyle{\frac 1{q'}}}) = \frac \theta 2,
$$
we get
$$
\frac \theta v+\frac 1{p_k} = \frac 1{p_0}+\frac 1{p_k} \ge \theta
$$
and
$$
\frac \theta v+\frac 1{q_k} = \frac 1{p_0}+\frac 1{q_k} \ge \theta
$$
when $k$ is odd. This implies that
\begin{equation}\label{Eq:LebExpEstOdd}
\frac 1{p_k},\frac 1{q_k} \ge \frac \theta {v'},\qquad k\in 2\mathbf Z+1.
\end{equation}

\par

By \eqref{Eq:LebExpEstEven} and \eqref{Eq:LebExpEstOdd},
there are $r,s\in [1,\infty ]^{N+1}$ such that
\eqref{Eq:InterpolCond1} and \eqref{Eq:InterpolCond2}
hold. We have
\begin{multline*}
(1-\theta ) \masfR _N({\textstyle{\frac 1r}}) = \masfR _N({\textstyle{\frac 1p}}) -\theta \masfR _N({\textstyle{\frac 1{v},\frac 1{v'},\dots ,\frac 1{v},\frac 1{v'}}})
\\[1ex]
\ge
\masfR _N({\textstyle{\frac 1{q'}}}) -\theta \left (
\frac 1{N-1}\left (
\frac {N+1}2 \cdot \left (
\frac 1{v'}+\frac 1v
\right )
-1
\right )
\right )
\\[1ex]
=
\frac \theta 2 -\theta \left (
\frac 1{N-1}\left (
\frac {N+1}2 \cdot \left (
\frac 1{v'}+\frac 1v
\right )
-1
\right )
\right )
=0
\end{multline*}
and
\begin{multline*}
(1-\theta ) \masfR _N({\textstyle{\frac 1{s'}}}) = \masfR _N({\textstyle{\frac 1{q'}}}) -\theta \masfR _N({\textstyle{\frac 1{v'},\frac 1{v},\dots ,\frac 1{v'},\frac 1{v}}})
\\[1ex]
=
\frac \theta 2 -\theta \left (
\frac 1{N-1}\left (
\frac {N+1}2 \cdot \left (
\frac 1{v'}+\frac 1v
\right )
-1
\right )
\right )
=0.
\end{multline*}
Consequently, if $p,q\in [1,\infty ]^{N+1}$ satisfy \eqref{Eq:pqConditionsA},
we have found $r,s,u\in [1,\infty ]^{N+1}$ 
and $\theta \in [0,1]$ such that
\eqref{Eq:WeylProdExprs}--\eqref{Eq:InterpolCond2} hold. Hence
the interpolation works out properly and the result follows.
\end{proof}

\par

Next we polish up Proposition \ref{Prop:MainPropOdd},
by purging away some superfluous conditions. More precisely
we have the following.

\par

\begin{thm}\label{Thm:MainThmOdd}
Let $N\ge 3$ be odd, $\masfR _N$ be as in \eqref{Eq:HYfunctional},
$\masfQ _{0,N}$ and $\masfQ _N$ be as in \eqref{Eq:sfQfunctional1},
and let $p_j,q_j\in [1,\infty ]$, $j=0,1,\dots , N$, be such that
\begin{equation}\label{Eq:pqConditionsB}
\max \left ( \masfR _N({\textstyle{\frac 1{q'}}}) ,0 \right )
\le  \min 
\left ( \masfQ _{N}({\textstyle{\frac 1p}}), \masfQ _{0,N}({\textstyle{\frac 1{q'}}}),
\masfQ _N({\textstyle{\frac 1p,\frac 1q}}),
\masfR _N({\textstyle{\frac 1p}})\right ).
\end{equation}
Also let $\omega _j \in \mascP _E(\rr {4d})$, $j=0,1,\dots ,N$, and
suppose \eqref{Eq:WeightCond} holds. 
Then the map \eqref{Eq:Weylmap} 
from $\maclS _{1/2}(\rr {2d}) \times \cdots \times
\maclS _{1/2}(\rr {2d})$ to $\maclS _{1/2}(\rr {2d})$
extends uniquely to a continuous and associative map from
$\splM ^{p_1,q_1} _{(\omega _1)}(\rr {2d}) \times \cdots
\times \splM ^{p_N,q_N} _{(\omega _N)}(\rr {2d})$ to
$\splM ^{p_0',q_0'} _{(1/\omega _0)}(\rr {2d})$.
\end{thm}

%

\par

We need some preparations for the proof of Theorem
\ref{Thm:MainThmOdd}. First we have
the following analogy of \cite[Lemma 2.7]{CoToWa}.

\par

\begin{lemma}\label{Lemma:Athm0.3}
Let $N \ge 3$ be odd, $x_j\in [0,1]$, $y_{j,k}=\frac 12(x_j+x_k)$ $j,k=0,\dots ,N$,
and consider the inequalities:
\begin{enumerate}
\item[{\rm{(1)}}] $\displaystyle{(N-1)^{-1}\left (\sum _{k=0}^Nx_k -1\right )\le 
\min _{(j,k)\in \Omega _N}y_{j,k}}$;

\vrum

\item[{\rm{(2)}}] $y_{j,k}\le \frac 12$, for all $(j,k)\in \Omega _N$;

\vrum

\item[{\rm{(3)}}] $\displaystyle{(N-1)^{-1}\left (\sum _{k=0}^Nx_k -1\right )
\le \min _{(j,k)\in \Omega _N}(1-y_{j,k})}$.
\end{enumerate}
Then
$$
{\rm{(1)}}\Rightarrow {\rm{(2)}} \Rightarrow {\rm{(3)}}.
$$
\end{lemma}

\par

\begin{rem}\label{Remark:Athm0.3}
We notice the similarities between the previous lemma and
\cite[Lemma 2.7]{CoToWa}. In fact, let
$N \ge 2$, $x_j\in [0,1]$, $j=0,\dots ,N$ and consider the inequalities:
\begin{enumerate}
\item[{\rm{(1)}}] $\displaystyle{(N-1)^{-1}\left (\sum _{k=0}^Nx_k -1\right )\le 
\min _{0\le j\le N}x_j}$;

\vrum

\item[{\rm{(2)}}] $x_j+x_k\le 1$, for all $k\neq j$;

\vrum

\item[{\rm{(3)}}] $\displaystyle{(N-1)^{-1}\left (\sum _{k=0}^Nx_k -1\right )
\le \min _{0\le j\le N}(1-x_j)}$.
\end{enumerate}
Then Lemma \cite[Lemma 2.7]{CoToWa} shows that 
$$
{\rm{(1)}}\Rightarrow {\rm{(2)}} \Rightarrow {\rm{(3)}}.
$$
\end{rem}

\par

\begin{proof}
We shall use similar ideas as in the proof of Lemma
\cite[Lemma 2.7]{CoToWa}.
Let
\begin{equation}\label{Eq:J1NJ2NDef}
J_{1,N}=\{ 0,\dots ,N \} \cap 2\mathbf Z,
\quad \text{and}\quad
J_{2,N}=\{ 0,\dots ,N \} \cap (2\mathbf Z +1),
\end{equation}
and assume that (1) holds but (2) fails. Then $x_j+x_k>1$
for some $(j,k)\in \Omega _N$.
By renumbering we may assume that $x_2+x_3\le x_j+x_{j+1}$
for every $j\in J_{1,N}$, and that $x_0+x_1>1$. Then (1) and the
fact that there are $(N-1)/2$ pairs of $(j,j+1)$ with
$j\in J_{1,N}\setminus 0$ give
\begin{multline*}
(N-1)y_{2,3} = \frac {N-1}2 (2y_{2,3})
\le
\sum _{m=1}^{(N-1)/2} 2y_{2m,2m+1}
\\[1ex]
= \sum _{j=2}^{N} x_j
< \sum _{j=0}^Nx_j -1 \le (N-1)y_{2,3},
\end{multline*}
which is a contradiction. Hence the assumption
$x_0+x_1>1$ must be wrong and it follows that (1) implies (2).

\par

Now suppose that (2) holds, and let
$j_0\in J_{1,N}$ and $k_0\in J_{2,N}$.
Then
$$
x_j\le 1-x_k
\quad \text{and}\quad
x_j\le 1-x_k ,\quad j\in J_{1,N},\ k\in J_{2,N}.
$$
This gives
\begin{equation*}
\sum _{j\in J_{1,N}\setminus j_0} x_j \le \frac {N-1}2(1-x_{k_0}) 
\quad \text{and}\quad
\sum _{k\in J_{2,N}\setminus k_0} x_k \le \frac {N-1}2(1-x_{j_0}) ,
\end{equation*}
giving that
$$
\sum _{j\in I_{N}\setminus \{ j_0,k_0\} }
x_j \le (N-1)\left (1-\frac 12 \left ( x_{j_0} +x_{k_0}  \right ) \right )
= (N-1)(1-y_{j_0,k_0}).
$$
Since
$$
x_{j_0}+x_{k_0}-1 = 2y_{j_0,k_0}-1 \le 0,
$$
we obtain
\begin{multline*}
\sum _{j\in I_{N}} x_j -1
=
(x_{j_0}+x_{k_0}-1) + \sum _{j\in I_{N}\setminus \{ j_0,k_0\} } x_j
\\[1ex]
\le
\sum _{j\in I_{N}\setminus \{ j_0,k_0\} } x_j \le (N-1)(1-y_{j_0,k_0}).
\end{multline*}
Since $j_0\in J_{1,N}$ and $k_0\in J_{2,N}$ was chosen arbitrary, it
follows that (3) holds.
\end{proof}

\par

\begin{proof}[Proof of Theorem \ref{Thm:MainThmOdd}]
%
%
If $I_N=\{ 0,1,\dots ,N\}$ as before and $j,k\in I_N$ satisfies
$j+k\in 2\mathbf Z+1$. Then the assumptions and Lemma
\ref{Lemma:Athm0.3} implies that
\begin{equation}\label{Eq:RNqLessMeans}
0\le \masfR _N({\textstyle{\frac 1{q'}}}) \le \min 
\left (
\frac 12 \left ( \frac 1{q_j'}+\frac 1{q_k'} \right ),
\frac 12 \left ( \frac 1{q_j}+\frac 1{q_k} \right )
\right ),
\qquad j+k\in 2\mathbf Z +1.
\end{equation}
Hence \eqref{Eq:pqConditionsB} implies \eqref{Eq:pqConditionsA},
and the result follows from Proposition \ref{Prop:MainPropOdd}.
\end{proof}

\par

\begin{rem}
We observe that Theorem \ref{Thm:MainThmOdd} implies that
the inclusion
\begin{equation}\label{Ex3-lin-form}
\splM ^{\infty ,1}\wpr \splM ^{2,2}\wpr \splM ^{2,2} \subseteq \splM ^{2,2}.
\end{equation}
In this context we observe that Theorem 0.1$'$ in \cite{CoToWa}
does ensure the validity of this inclusion, while Theorem 2.9 in
\cite{CoToWa} does.
\end{rem}

%

%
%
%
%
%
%

\par

\par

We may use \eqref{calculitransform} and Proposition
\ref{propCalculiTransfMod} to extend Theorem \ref{Thm:MainThmOdd}
 to involve more general products
arising in the pseudo-differential calculi. More precisely, the let
$\GL (d,\Omega)$ be the set of all $d\times d$ matrices with entries in
the set $\Omega$, and let $A\in \GL (d,\mathbf R)$.
By \eqref{calculitransform} we have
\begin{multline*}
a _1 \wpr _{\! A} \cdots \wpr _{\! A} a_N
=
e^{-i\scal {A_0D_\xi }{D_x}} ((e^{i\scal {A_0D_\xi }{D_x}}a _1)
\wpr  \cdots \wpr  (e^{i\scal {A_0D_\xi }{D_x}} a_N)),
\\[1ex]
A_0=A-\frac 12 I_d,
\end{multline*}
where $I_d$ is the $d\times d$ unit matrix (see (2.14) and (2.15)
in \cite{Toft15}).
If we combine this relation with
Proposition \ref{propCalculiTransfMod} and Theorem
\ref{Thm:MainThmOdd}, we get the following result.
The condition on the weight functions is
\begin{multline}\label{weightcondtcalc}
1 \lesssim \omega _0(T_A(X_N,X_0))\prod _{j=1}^N
\omega _j(T_A(X_{j},X_{j-1})),
\quad X_0,\dots ,X_N \in \rr {2d},
\end{multline}
where
\begin{multline}\label{Ttdef}
T_A(X,Y) =(y+A(x-y),\xi +A^*(\eta -\xi ),\eta -\xi , x-y),
\\[1ex]
X=(x,\xi )\in \rr {2d},\ Y=(y,\eta )\in \rr {2d}.
\end{multline}
(See (2.16) and (2.17) in \cite{Toft15}.)

%
%
%

\par

\begin{thm}\label{Thm:MainThmOdd2}
Let $A\in \GL (d,\mathbf R)$, $N\ge 3$ be odd, $\masfR _N$
be as in \eqref{Eq:HYfunctional}, $\masfQ _{0,N}$ and $\masfQ _N$
be as in \eqref{Eq:sfQfunctional1}, and let
$p_j,q_j\in [1,\infty ]$, $j=0,1,\dots , N$, be such that
\begin{equation}\label{Eq:pqConditionsBAgain}
\max \left ( \masfR _N({\textstyle{\frac 1{q'}}}) ,0 \right )
\le  \min 
\left ( \masfQ _{N}({\textstyle{\frac 1p}}), \masfQ _{0,N}({\textstyle{\frac 1{q'}}}),
\masfQ _N({\textstyle{\frac 1p,\frac 1q}}),
\masfR _N({\textstyle{\frac 1p}})\right ).
\end{equation}
Also let $\omega _j \in \mascP _E(\rr {4d})$, $j=0,1,\dots ,N$,
and suppose
\eqref{weightcondtcalc} and \eqref{Ttdef} hold. Then the map
\eqref{Eq:Weylmap}$'$ from $\maclS _{1/2}(\rr {2d}) \times \cdots \times
\maclS _{1/2}(\rr {2d})$ to $\maclS _{1/2}(\rr {2d})$
extends uniquely to a continuous and associative map from $M ^{p_1,q_1}
_{(\omega _1)}(\rr {2d}) \times \cdots \times M ^{p_N,q_N}
_{(\omega _N)}(\rr {2d})$ to $M ^{p_0',q_0'} _{(1/\omega _0)}(\rr {2d})$.
\end{thm}

\par

In the same way we get the following result by combining
\eqref{calculitransform} and Propositions
\ref{propCalculiTransfMod} and \ref{Prop:Prop2}. The details
are left for the reader.

\par

\begin{prop}\label{Prop:Prop2Ext}
Let $A\in \GL (d,\mathbf R)$, $N\ge 3$ be odd,
$p,p_j\in (0,\infty ]$, $j=1,\dots , N$,
and let
$\omega _j \in \mascP _E(\rr {4d})$, $j=0,1,\dots ,N$, and
suppose \eqref{weightcondtcalc} and \eqref{Ttdef} hold. 
Then the following is true:
\begin{enumerate}
\item[{\rm{(1)}}] if $p_0=p_N=p$, $p_j=\max (1,p)$ when $j\in [3,N-2]$ is odd
and $p_j=p'$ when $j$ is even, then the map \eqref{Eq:Weylmap}$'$ 
from $\maclS _{1/2}(\rr {2d}) \times \cdots \times
\maclS _{1/2}(\rr {2d})$ to $\maclS _{1/2}(\rr {2d})$
extends uniquely to a continuous and associative map from
$\splM ^{p_1} _{(\omega _1)}(\rr {2d}) \times \cdots
\times \splM ^{p_N} _{(\omega _N)}(\rr {2d})$ to
$\splM ^{p} _{(1/\omega _0)}(\rr {2d})$;

\par

\item[{\rm{(1)}}] if $p_j=p$ when $j$ is even and
$p_j=p'$ when is odd,
then the map \eqref{Eq:Weylmap}$'$
from $\maclS _{1/2}(\rr {2d}) \times \cdots \times
\maclS _{1/2}(\rr {2d})$ to $\maclS _{1/2}(\rr {2d})$
extends uniquely to a continuous and associative map from
$\splM ^{p_1} _{(\omega _1)}(\rr {2d}) \times \cdots
\times \splM ^{p_N} _{(\omega _N)}(\rr {2d})$ to
$\splM ^{p'} _{(1/\omega _0)}(\rr {2d})$.
\end{enumerate}
\end{prop}

\par

Finally we prove a continuity result for the twisted convolution. 
The map \eqref{Eq:Weylmap} is then replaced by
\begin{equation}\label{Twistmap}
(a_1,a_2,\dots ,a_N)\mapsto a_1*_\sigma a_2*_\sigma \cdots *_\sigma a_N.
\end{equation}
The following
result follows immediately from Theorem \ref{Thm:MainThmOdd}.
Here the condition \eqref{Eq:WeightCond} is replaced by
\begin{multline}\label{weightcond3}
1 \lesssim \omega _0(X_N-X_0,X_N+X_0)\prod _{j=1}^N
\omega _j(X_j-X_{j-1},X_j+X_{j-1}),
\\[1ex]
X_0,X_1,\dots ,X_N\in \rr {2d}.
\end{multline}

\par

\begin{thm}\label{Thm:MainThmOddTwistConv}
Let $p_j,q_j\in [1,\infty ]$, $j=0,1,\dots , N$, and suppose that
\begin{equation*} 
\max \left ( \masfR _N({\textstyle{\frac 1{p'}}}) ,0 \right )
\le  \min 
\left ( \masfQ _{N}({\textstyle{\frac 1q}}), \masfQ _{0,N}({\textstyle{\frac 1{p'}}}),
\masfQ _N({\textstyle{\frac 1p,\frac 1q}}),
\masfR _N({\textstyle{\frac 1q}})\right ).
\end{equation*}
%
%
Suppose $\omega _j\in \mascP _E(\rr {4d})$, $j=0,1,\dots ,N$, satisfy
\eqref{weightcond3}. Then the map
\eqref{Twistmap} from $\maclS _{1/2}(\rr {2d}) \times \cdots \times
\maclS _{1/2}(\rr {2d})$ to $\maclS _{1/2}(\rr {2d})$
extends uniquely to a continuous and associative map from
$\splW^{p_1,q_1}_{(\omega _1)}(\rr {2d})
\times \cdots \times \splW^{p_N,q_N}_{(\omega _N)}(\rr {2d})$
to $\splW^{p_0',q_0'} _{(1/\omega _0)}(\rr {2d})$.
\end{thm}

\par

\par


\end{document}